\documentclass[12pt]{amsart}
\usepackage{amssymb}
\usepackage{t1enc}
\usepackage[latin2]{inputenc}
\usepackage{verbatim}
\usepackage{amsmath,amsfonts,amssymb,amsthm}
\usepackage[mathcal]{eucal}
\usepackage{enumerate}
\usepackage[centertags]{amsmath}
\usepackage{graphics}

\newtheorem{theorem}{Theorem}
\newtheorem{lemma}{Lemma}

\newtheorem{corollary}{Corollary}

\numberwithin{equation}{subsection}

 \DeclareMathOperator{\mes}{mes}

\begin{document}
\author{Uhangi Goginava}
\title[Convergence  of logarithmic means]{Convergence  of logarithmic means of quadratical partial sums of double Fourier
series}
\address{U. Goginava, Department of Mathematics, Faculty of Exact and
Natural Sciences, Iv. Javakhishvili Tbilisi State University, Chavchavadze
str. 1, Tbilisi 0128, Georgia}
\email{zazagoginava@gmail.com}
\maketitle

\begin{abstract}
In this paper we investigate some convergence and divergence properties of the
logarithmic means of quadratical partial sums of double Fourier series of
functions in the measure and in the $L$ Lebesgue norm.
\end{abstract}

\footnotetext{%
2010 Mathematics Subject Classification 42A24 .
}

\section{Introduction}

In the literature, it is known the notion of the Riesz's logarithmic means
of a Fourier series. The $n$-th mean of the Fourier series of the integrable
function $f$ is defined by
\begin{equation*}
\frac{1}{l_{n}}\sum_{k=1}^{n-1}\frac{S_{k}(f)}{k},
\end{equation*}%
where $S_{k}(f)$ is the partial sum of its Fourier series. This Riesz's
logarithmic means with respect to the trigonometric system has been studied
by a lot of authors. We mention for instance the papers of Szász, and Yabuta
\cite{sza, ya}. This mean with respect to the Walsh, Vilenkin system is
discussed by Simon, and Gát \cite{Sim, gat}.

Let $\left\{ q_{k}:k\geq 0\right\} $ be a sequence of nonnegative numbers.
The Nörlund means for the Fourier series of $f$ are defined by
\begin{equation*}
\frac{1}{\sum_{k=1}^{n}q_{k}}\sum_{k=0}^{n-1}q_{n-k}S_{k}(f).
\end{equation*}
If $q_{k}=\frac{1}{k}$, then we get the (Nörlund) logarithmic means:
\begin{equation*}
\frac{1}{l_{n}}\sum_{k=0}^{n-1}\frac{S_{k}(f)}{n-k},l_{n}:=\sum_{k=1}^{n}%
\frac{1}{k}.
\end{equation*}%
In this paper we call it logarithmic means. Although, it is a kind of
\textquotedblleft reverse\textquotedblright\ Riesz's logarithmic means. In
\cite{gg} we proved some convergence and divergence properties of the
logarithmic means of Walsh-Fourier series of functions in the class of
continuous functions, and in the Lebesgue space $L$. \ In this paper we
discuss some convergence and divergence properties of logarithmic means of
quadratical partial sums of the double Fourier series of functions in the $L$
Lebesgue norm (see Theorems \ref{normconvergence} and \ref{normdivergence} ).

The partial sums $S_{n}(f)$ of the Fourier series of a function $f\in
L(T),\quad T=[-\pi ,\pi )$ converges in measure on $T$. The condition $f\in
L\log L(T^{2})$ provides convergence in measure on $T^{2}$ of the
rectangular partial sums $S_{n,m}(f)$ of double Fourier series \cite{Zh}.
The first example of a function from classes wider than $L\log L(T^{2})$
with $S_{n,n}(f)$ divergent in measure on $T^{2}$ was obtained in Getsadze
\cite{Ge} and Konyagin \cite{Kon}.

In the present paper we investigate convergence in measure of logarithmic
means of quadratical partial sums
\begin{equation*}
\frac{1}{l_{n}}\sum_{i=0}^{n-1}\frac{S_{i,i}(f,x,y)}{n-i}
\end{equation*}%
of double Fourier series and prove Theorem \ref{measuredivergence}, that is,
for any Orlicz space, which is not a subspace of $L\log L(I^{2})$, the set
of the functions that these means converges in measure is of first Baire
category. From this result follows that (Corollary \ref{measuredivergence2})
in classes wider than $L\log L(T^{2})$ there exists functions $f$ for which
logarithmic means $t_{n}\left( f\right) $ of quadratical partial sums of
double Fourier series diverges in measure. Besides, it is surprising that
the two cases (the logarithmic means of quadratical and the two-dimensional
partial sums) are not different in this point of view. Namely, for instance
in the case of $(C,1)$ means we have a quite different situation. That is,
it is well-known \cite{Zh} that the Marcinkiewicz means $\sigma _{n}\left(
f\right) =\frac{1}{n}\sum\limits_{j=1}^{n}S_{j,j}\left( f\right) $, that is
the $(C,1)$ means of quadratical partial sums of double trigonometric
Fourier series of a function $f\in L$ converges in $L$-norm and a.e. to $f$.
Thus, in question of convergence in measure and in norm logarithmic means of
quadratical partial sums of double Fourier series differs from Marcinkiewicz
means and like the usual quadratical partial sums of double Fourier series.

The results for summability of quadratical partial sums of Walsh-Fourier
series can be found in \cite{GogJAT, GGNSMH}

\section{Definitions and Notation}

We denote by $L_{0}=L_{0}(T^{2})$ the Lebesque space of functions that are
measurable and finite almost everywhere on $T^{2}=[-\pi ,\pi )\times \lbrack
-\pi ,\pi )$. $\mes(A)$ is the Lebesque measure of the set $A\subset T^{2}$.

Let $L_{Q}=L_{Q}(T^{2})$ be the Orlicz space \cite{K-R} generated by Young
function $Q$, i.e. $Q$ is convex continuous even function such that $Q(0)=0$
and

\begin{equation*}
\lim\limits_{u\rightarrow +\infty }\frac{Q\left( u\right) }{u}=+\infty
,\,\,\,\,\lim\limits_{u\rightarrow 0}\frac{Q\left( u\right) }{u}=0.
\end{equation*}

This space is endowed with the norm
\begin{equation*}
\Vert f\Vert _{L_{Q}(T^{2})}=\inf \{k>0:\iint\limits_{T^{2}}Q(\left\vert
f(x,y)\right\vert /k)dxdy\leq 1\}.
\end{equation*}

In particular, if $Q(u)=u\log (1+u)$ and $Q(u)=u\log ^{2}(1+u)\quad u>0$,
then the corresponding space will be denoted by $L\log L(T^{2})$ and $L\log
^{2}L(T^{2})$, respectively.

Let $f\in L_{1}\left( T^{2}\right) .$ The Fourier series of $f$ with respect
to the trigonometric system is the series
\begin{equation*}
S\left[ f\right] :=\sum_{m,n=-\infty }^{+\infty }\widehat{f}\left(
m,n\right) e^{imx}e^{iny},
\end{equation*}%
where
\begin{equation*}
\widehat{f}\left( m,n\right) =\frac{1}{4\pi ^{2}}\iint%
\limits_{T^{2}}f(x,y)e^{-imx}e^{-iny}dxdy
\end{equation*}%
are the Fourier coefficients of the function $f$. The rectangular partial
sums are defined as follows:
\begin{equation*}
S_{M,N}(f;x,y):=\sum_{m=-M}^{M}\sum_{n=-N}^{N}\widehat{f}\left( m,n\right)
e^{imx}e^{iny}.
\end{equation*}

The logarithmic means of quadratical partial sums of double Fourier series
is defined as follows
\begin{equation*}
t_{n}\left( f,x,y\right) =\frac{1}{l_{n}}\sum\limits_{i=0}^{n-1}\frac{%
S_{i,i}\left( f,x,y\right) }{n-i}.
\end{equation*}

It is evident that
\begin{equation*}
t_{n}\left( f,x,y\right) =\iint\limits_{T^{2}}f\left( s,t\right) F_{n}\left(
x-s,y-t\right) dsdt,
\end{equation*}%
where
\begin{equation*}
F_{n}\left( t,s\right) =\frac{1}{l_{n}}\sum\limits_{k=0}^{n-1}\frac{%
D_{k}\left( t\right) D_{k}\left( s\right) }{n-k}.
\end{equation*}

\section{Main Results}

It is well-known that the following theorem is true.

\begin{theorem}
\label{quadratical}Let $f\in L\log ^{2}L\left( T^{2}\right) $. Then
\begin{equation*}
\left\Vert S_{nn}\left( f\right) -f\right\Vert _{L_{1}\left( T^{2}\right)
}\rightarrow 0\text{ \ as \ }n\rightarrow \infty \text{.}
\end{equation*}
\end{theorem}

Due to the inequality%
\begin{equation*}
\left\Vert t_{n}\left( f\right) -f\right\Vert _{L_{1}\left( T^{2}\right)
}\leq \frac{1}{l_{n}}\sum\limits_{k=0}^{n-1}\frac{\left\Vert S_{kk}\left(
f\right) -f\right\Vert _{L_{1}\left( T^{2}\right) }}{n-k}
\end{equation*}%
and the fact that the Nörlund logarithmic summability method is regular (%
\cite{Ha}, Ch. 3) Theorem \ref{quadratical} yield

\begin{theorem}
\label{normconvergence}Let $f\in L\log ^{2}L\left( T^{2}\right) $. Then%
\begin{equation*}
\frac{1}{l_{n}}\sum\limits_{k=0}^{n-1}\frac{\left\Vert S_{kk}\left( f\right)
-f\right\Vert _{L_{1}\left( T^{2}\right) }}{n-k}\rightarrow 0\text{ \ as \ }%
n\rightarrow \infty \text{.}
\end{equation*}
\end{theorem}

In this paper we investigate the sharpness of Theorem \ref{normconvergence}.
Moreover, we prove

\begin{theorem}
\label{normdivergence}Let $L_{Q}\left( T^{2}\right) $ be an Orlicz space,
such that%
\begin{equation*}
L_{Q}\left( T^{2}\right) \nsubseteq L\log ^{2}L(T^{2}).
\end{equation*}%
Then \newline
\newline
a)
\begin{equation*}
\sup\limits_{n}\left\Vert t_{n}\right\Vert _{L_{Q}\left( T^{2}\right)
\rightarrow L_{1}\left( T^{2}\right) }=+\infty ;
\end{equation*}%
\newline
\newline
b) there exists a function $f\in L_{Q}\left( T^{2}\right) $ such that $%
t_{n}\left( f\right) $ does not converge to $f$ in $L_{1}\left( T^{2}\right)
$-norm.
\end{theorem}

\begin{theorem}
\label{measuredivergence}Let $L_{Q}(T^{2})$ be an Orlicz space, such that

\begin{equation*}
L_{Q}(T^{2})\nsubseteq L\log L(T^{2}).
\end{equation*}

Then the set of the functions from the Orlicz space $L_{Q}(T^{2})$ with
logarithmic means of quadratical partial sums of double Fourier series
convergent in measure on $T^{2}$ is of first Baire category in $%
L_{Q}(T^{2}). $
\end{theorem}

\begin{corollary}
\label{measuredivergence2}Let $\varphi :[0,\infty \lbrack \rightarrow
\lbrack 0,\infty \lbrack $ be a nondecreasing function satisfying for $%
x\rightarrow +\infty $ the condition

\begin{equation*}
\varphi (x)=o(x\log x).
\end{equation*}

Then there exists the function $f\in L_{1}(T^{2})$ such that

a)
\begin{equation*}
\iint\limits_{T^{2}}\varphi (\left\vert f(x,y)\right\vert )dxdy<\infty ;
\end{equation*}

b) logarithmic means of quadratical partial sums of double Fourier series of
$f$ diverges in measure on $T^{2}$.
\end{corollary}

\section{Auxiliary Results}

We apply the reasoning of \cite{G} formulated as the following proposition
in particular case.

\begin{theorem}[Garsia]
\label{Garsia}  Let $H:L_{1}(T^{2})\rightarrow L_{0}(T^{2})$ be a linear
continuous operator, which commutes with family of translations $\mathcal{E}$%
, i. e. $\forall E\in \mathcal{E}\quad \forall f\in L_{1}(T^{2})\quad HEf=EHf
$. Let $\Vert f\Vert _{L_{1}(T^{2})}=1$ and $\lambda >1$. 
Then for any $1\leq r\in \mathbb{N}$ under condition mes$\{\left( x,y\right)
\in T^{2}:|Hf|>\lambda \}\geq \frac{1}{r}$ there exist $%
E_{1},...,E_{r},E_{1}^{\prime },...,E_{r}^{\prime }\in \mathcal{E}$ and $%
\varepsilon _{i}=\pm 1,\quad i=1,...,r$ such that
\begin{equation*}
\mes\{\left( x,y\right) \in T^{2}:|H(\sum_{i=1}^{r}\varepsilon
_{i}f(E_{i}x,E_{i}^{\prime }y)|>\lambda \}\geq \frac{1}{8}.
\end{equation*}
\end{theorem}

\begin{lemma}
\label{GGT}Let $\{H_{m}\}_{m=1}^{\infty }$ be a sequence of linear continues
operators, acting from Orlicz space $L_{Q}(T^{2})$ in to the space $%
L_{0}(T^{2})$. Suppose that there exists the sequence of functions $\{\xi
_{k}\}_{k=1}^{\infty }$ from unit bull $S_{Q}(0,1)$ of space $L_{Q}(T^{2})$,
sequences of integers $\{m_{k}\}_{k=1}^{\infty }$ and $\{\nu
_{k}\}_{k=1}^{\infty }$ increasing to infinity such that

\begin{equation*}
\varepsilon _{0}=\inf_{k}\mes\{\left( x,y\right) \in T^{2}:|H_{m_{k}}\xi
_{k}\left( x,y\right) |>\nu _{k}\}>0.
\end{equation*}

Then $B$ - the set of functions $f$ from space $L_{Q}(T^{2})$, for which the
sequence $\{H_{m}f\}$ converges in measure to an a. e. finite function is of
first Baire category in space $L_{Q}(T^{2})$.
\end{lemma}

The proof of Lemma \ref{GGT} can be found in \cite{GGT}.

\begin{lemma}
\label{GGT2}Let $L_{\Phi }\left( T^{2}\right) $ be an Orlicz space and let $%
\varphi :[0,\infty )\rightarrow \lbrack 0,\infty )$ be measurable function
with condition $\varphi \left( x\right) =o\left( \Phi \left( x\right)
\right) $ as $x\rightarrow \infty .$ Then there exists Orlicz space $%
L_{\omega }\left( T^{2}\right) $, such that $\omega \left( x\right) =o\left(
\Phi \left( x\right) \right) \,\,$ as $x\rightarrow \infty $, and $\omega
\left( x\right) \geq \varphi \left( x\right) $ for $x\geq c\geq 0.$
\end{lemma}

The proof of Lemma \ref{GGT2} can be found in \cite{GGT2}.

\begin{lemma}
\label{main} Let%
\begin{equation*}
\alpha _{mn}:=\frac{\arccos \left( 1/4\right) +2\pi m}{2^{2n}+1/2},\beta
_{mn}:=\frac{\pi /2+2\pi m}{2^{2n}+1/2},n,m=0,1,....
\end{equation*}%
Then there exists a posotive integer $n_{0}$, such that
\begin{equation*}
F_{2^{2n}}\left( x,y\right) \geq \frac{c}{xy},n\geq n_{0},
\end{equation*}%
where
\begin{equation*}
\left( x,y\right) \in I_{n}:=\bigcup\limits_{l,m=1}^{2^{n-3}}\left\{ \left(
x,y\right) :\alpha _{mn}\leq x\leq \beta _{mn},\alpha _{ln}\leq x\leq \beta
_{ln}\right\} .
\end{equation*}
\end{lemma}

\begin{proof}[Proof of Lemma \protect\ref{main}]
We can write%
\begin{equation}
l_{2^{2n}}F_{2^{2n}}\left( x,y\right) =\sum\limits_{k=1}^{2^{2n}}\frac{%
D_{2^{2n}-k}\left( x\right) D_{2^{2n}-k}\left( y\right) }{k}  \label{eq}
\end{equation}%
\begin{equation*}
=\sum\limits_{k=1}^{2^{2n}}\frac{1}{k}\frac{\sin \left( 2^{2n}+1/2-k\right) x%
}{2\sin \frac{x}{2}}\frac{\sin \left( 2^{2n}+1/2-k\right) y}{2\sin \frac{y}{2%
}}
\end{equation*}%
\begin{equation*}
=\frac{\sin \left( 2^{2n}+1/2\right) x}{2\sin \frac{x}{2}}\frac{\sin \left(
2^{2n}+1/2\right) y}{2\sin \frac{y}{2}}\sum\limits_{k=1}^{2^{2n}}\frac{\cos
kx\cos ky}{k}
\end{equation*}%
\begin{equation*}
+\frac{\cos \left( 2^{2n}+1/2\right) x}{2\sin \frac{x}{2}}\frac{\cos \left(
2^{2n}+1/2\right) y}{2\sin \frac{y}{2}}\sum\limits_{k=1}^{2^{2n}}\frac{\sin
kx\sin ky}{k}
\end{equation*}%
\begin{equation*}
-\frac{\sin \left( 2^{2n}+1/2\right) x}{2\sin \frac{x}{2}}\frac{\cos \left(
2^{2n}+1/2\right) y}{2\sin \frac{y}{2}}\sum\limits_{k=1}^{2^{2n}}\frac{\cos
kx\sin ky}{k}
\end{equation*}%
\begin{equation*}
-\frac{\cos \left( 2^{2n}+1/2\right) x}{2\sin \frac{x}{2}}\frac{\sin \left(
2^{2n}+1/2\right) y}{2\sin \frac{y}{2}}\sum\limits_{k=1}^{2^{2n}}\frac{\sin
kx\cos ky}{k}
\end{equation*}%
\begin{equation*}
=\frac{\sin \left( 2^{2n}+1/2\right) x}{2\sin \frac{x}{2}}\frac{\sin \left(
2^{2n}+1/2\right) y}{2\sin \frac{y}{2}}
\end{equation*}%
\begin{equation*}
\times \frac{1}{2}\left\{ \sum\limits_{k=1}^{2^{2n}}\frac{\cos k\left(
x+y\right) }{k}+\sum\limits_{k=1}^{2^{2n}}\frac{\cos k\left( x-y\right) }{k}%
\right\}
\end{equation*}%
\begin{equation*}
+\frac{\cos \left( 2^{2n}+1/2\right) x}{2\sin \frac{x}{2}}\frac{\cos \left(
2^{2n}+1/2\right) y}{2\sin \frac{y}{2}}
\end{equation*}%
\begin{equation*}
\times \frac{1}{2}\left\{ \sum\limits_{k=1}^{2^{2n}}\frac{\cos k\left(
x-y\right) }{k}-\sum\limits_{k=1}^{2^{2n}}\frac{\cos k\left( x+y\right) }{k}%
\right\}
\end{equation*}%
\begin{equation*}
-\frac{\sin \left( 2^{2n}+1/2\right) x}{2\sin \frac{x}{2}}\frac{\cos \left(
2^{2n}+1/2\right) y}{2\sin \frac{y}{2}}
\end{equation*}%
\begin{equation*}
\times \frac{1}{2}\left\{ \sum\limits_{k=1}^{2^{2n}}\frac{\sin k\left(
x+y\right) }{k}-\sum\limits_{k=1}^{2^{2n}}\frac{\sin k\left( x-y\right) }{k}%
\right\}
\end{equation*}%
\begin{equation*}
-\frac{\cos \left( 2^{2n}+1/2\right) x}{2\sin \frac{x}{2}}\frac{\sin \left(
2^{2n}+1/2\right) y}{2\sin \frac{y}{2}}
\end{equation*}%
\begin{equation*}
\times \frac{1}{2}\left\{ \sum\limits_{k=1}^{2^{2n}}\frac{\sin k\left(
x+y\right) }{k}+\sum\limits_{k=1}^{2^{2n}}\frac{\sin k\left( x-y\right) }{k}%
\right\} .
\end{equation*}

Since%
\begin{eqnarray*}
&&\sum\limits_{k=1}^{2^{2n}}\frac{\cos ku}{k} \\
&=&\sum\limits_{k=1}^{2^{2n}-2}\frac{2}{k\left( k+1\right) \left( k+2\right)
}\frac{\sin ^{2}\left( \left( k+1\right) \frac{u}{2}\right) }{2\sin ^{2}%
\frac{u}{2}} \\
&&-\frac{1}{2^{2n}\left( 2^{2n}-1\right) }\frac{\sin ^{2}\left( 2^{2n}\frac{u%
}{2}\right) }{2\sin ^{2}\frac{u}{2}}+\frac{1}{2^{2n}}\frac{\sin \left(
2^{2n}+1/2\right) u}{2\sin \frac{u}{2}}-\frac{3}{4}
\end{eqnarray*}%
from (\ref{eq}) \ we have%
\begin{equation}
l_{2^{2n}}F_{2^{2n}}\left( x,y\right) =\frac{\sin \left( 2^{2n}+1/2\right) x%
}{2\sin \frac{x}{2}}\frac{\sin \left( 2^{2n}+1/2\right) y}{2\sin \frac{y}{2}}%
\frac{1}{2}  \label{R1-R15}
\end{equation}%
\begin{equation*}
\times \sum\limits_{k=1}^{2^{2n}-2}\frac{2}{k\left( k+1\right) \left(
k+2\right) }\frac{\sin ^{2}\left( \left( k+1\right) \frac{x+y}{2}\right) }{%
2\sin ^{2}\frac{x+y}{2}}
\end{equation*}%
\begin{equation*}
+\frac{\sin \left( 2^{2n}+1/2\right) x}{2\sin \frac{x}{2}}\frac{\sin \left(
2^{2n}+1/2\right) y}{2\sin \frac{y}{2}}\frac{1}{2}
\end{equation*}%
\begin{equation*}
\times \sum\limits_{k=1}^{2^{2n}-2}\frac{2}{k\left( k+1\right) \left(
k+2\right) }\frac{\sin ^{2}\left( \left( k+1\right) \frac{x-y}{2}\right) }{%
2\sin ^{2}\frac{x-y}{2}}
\end{equation*}%
\begin{equation*}
+\frac{\cos \left( 2^{2n}+1/2\right) x}{2\sin \frac{x}{2}}\frac{\cos \left(
2^{2n}+1/2\right) y}{2\sin \frac{y}{2}}\frac{1}{2}
\end{equation*}%
\begin{equation*}
\times \sum\limits_{k=1}^{2^{2n}-2}\frac{2}{k\left( k+1\right) \left(
k+2\right) }\frac{\sin ^{2}\left( \left( k+1\right) \frac{x-y}{2}\right) }{%
2\sin ^{2}\frac{x-y}{2}}
\end{equation*}%
\begin{equation*}
-\frac{\cos \left( 2^{2n}+1/2\right) x}{2\sin \frac{x}{2}}\frac{\cos \left(
2^{2n}+1/2\right) y}{2\sin \frac{y}{2}}\frac{1}{2}
\end{equation*}%
\begin{equation*}
\times \sum\limits_{k=1}^{2^{2n}-2}\frac{2}{k\left( k+1\right) \left(
k+2\right) }\frac{\sin ^{2}\left( \left( k+1\right) \frac{x+y}{2}\right) }{%
2\sin ^{2}\frac{x+y}{2}}
\end{equation*}%
\begin{equation*}
-\frac{\sin \left( 2^{2n}+1/2\right) x}{2\sin \frac{x}{2}}\frac{\sin \left(
2^{2n}+1/2\right) y}{2\sin \frac{y}{2}}\frac{1}{2}
\end{equation*}%
\begin{equation*}
\times \frac{1}{2^{2n}\left( 2^{2n}-1\right) }\frac{\sin ^{2}\left( 2^{2n}%
\frac{x+y}{2}\right) }{2\sin ^{2}\frac{x+y}{2}}
\end{equation*}

\begin{equation*}
+\frac{\sin \left( 2^{2n}+1/2\right) x}{2\sin \frac{x}{2}}\frac{\sin \left(
2^{2n}+1/2\right) y}{2\sin \frac{y}{2}}\frac{1}{2}\frac{1}{2^{2n}}\frac{\sin
\left( \left( 2^{2n}+1/2\right) \left( x+y\right) \right) }{2\sin \frac{x+y}{%
2}}
\end{equation*}%
\begin{equation*}
-\frac{\sin \left( 2^{2n}+1/2\right) x}{2\sin \frac{x}{2}}\frac{\sin \left(
2^{2n}+1/2\right) y}{2\sin \frac{y}{2}}\frac{3}{2}
\end{equation*}%
\begin{equation*}
-\frac{\sin \left( 2^{2n}+1/2\right) x}{2\sin \frac{x}{2}}\frac{\sin \left(
2^{2n}+1/2\right) y}{2\sin \frac{y}{2}}\frac{1}{2^{2n+1}\left(
2^{2n}-1\right) }\frac{\sin ^{2}\left( 2^{2n}\frac{x-y}{2}\right) }{2\sin
^{2}\frac{x-y}{2}}
\end{equation*}%
\begin{equation*}
+\frac{\sin \left( 2^{2n}+1/2\right) x}{2\sin \frac{x}{2}}\frac{\sin \left(
2^{2n}+1/2\right) y}{2\sin \frac{y}{2}}\frac{1}{2^{2n+1}}\frac{\sin \left(
\left( 2^{2n}+1/2\right) \left( x-y\right) \right) }{2\sin \frac{x-y}{2}}
\end{equation*}%
\begin{equation*}
-\frac{\cos \left( 2^{2n}+1/2\right) x}{2\sin \frac{x}{2}}\frac{\cos \left(
2^{2n}+1/2\right) y}{2\sin \frac{y}{2}}\frac{1}{2^{2n+1}\left(
2^{2n}-1\right) }\frac{\sin ^{2}\left( 2^{2n}\frac{x-y}{2}\right) }{2\sin
^{2}\frac{x-y}{2}}
\end{equation*}%
\begin{equation*}
+\frac{\cos \left( 2^{2n}+1/2\right) x}{2\sin \frac{x}{2}}\frac{\cos \left(
2^{2n}+1/2\right) y}{2\sin \frac{y}{2}}\frac{1}{2^{2n+1}}\frac{\sin \left(
\left( 2^{2n}+1/2\right) \left( x-y\right) \right) }{2\sin \frac{x-y}{2}}
\end{equation*}%
\begin{equation*}
-\frac{\cos \left( 2^{2n}+1/2\right) x}{2\sin \frac{x}{2}}\frac{\cos \left(
2^{2n}+1/2\right) y}{2\sin \frac{y}{2}}\frac{1}{2^{2n+1}\left(
2^{2n}-1\right) }\frac{\sin ^{2}\left( 2^{2n}\frac{x+y}{2}\right) }{2\sin
^{2}\frac{x+y}{2}}
\end{equation*}%
\begin{equation*}
+\frac{\cos \left( 2^{2n}+1/2\right) x}{2\sin \frac{x}{2}}\frac{\cos \left(
2^{2n}+1/2\right) y}{2\sin \frac{y}{2}}\frac{1}{2^{2n+1}}\frac{\sin \left(
\left( 2^{2n}+1/2\right) \left( x+y\right) \right) }{2\sin \frac{x+y}{2}}
\end{equation*}%
\begin{equation*}
-\frac{\sin \left( 2^{2n}+1/2\right) x}{2\sin \frac{x}{2}}\frac{\cos \left(
2^{2n}+1/2\right) y}{2\sin \frac{y}{2}}
\end{equation*}%
\begin{equation*}
\times \frac{1}{2}\left\{ \sum\limits_{k=1}^{2^{2n}}\frac{\sin k\left(
x+y\right) }{k}-\sum\limits_{k=1}^{2^{2n}}\frac{\sin k\left( x-y\right) }{k}%
\right\}
\end{equation*}%
\begin{equation*}
-\frac{\cos \left( 2^{2n}+1/2\right) x}{2\sin \frac{x}{2}}\frac{\sin \left(
2^{2n}+1/2\right) y}{2\sin \frac{y}{2}}
\end{equation*}%
\begin{equation*}
\times \frac{1}{2}\left\{ \sum\limits_{k=1}^{2^{2n}}\frac{\sin k\left(
x+y\right) }{k}+\sum\limits_{k=1}^{2^{2n}}\frac{\sin k\left( x-y\right) }{k}%
\right\}
\end{equation*}%
\begin{equation*}
:=\sum\limits_{j=1}^{15}R_{j}.
\end{equation*}

Since%
\begin{equation*}
\left\vert \sum\limits_{k=1}^{2^{2n}}\frac{\sin kx}{k}\right\vert \leq
c<\infty ,n=1,2,..
\end{equation*}%
and%
\begin{equation}
\left\vert \sin x\right\vert \geq \frac{2}{\pi }\left\vert x\right\vert ,-%
\frac{\pi }{2}\leq x\leq \frac{\pi }{2}  \label{low}
\end{equation}%
we obtain%
\begin{equation}
\sum\limits_{j=5}^{15}\left\vert R_{j}\right\vert =O\left( \frac{1}{xy}%
\right) .  \label{R5-R15}
\end{equation}

Let $\left( x,y\right) \in I_{n}$. Then it is easy to show that%
\begin{equation*}
\sin \left( 2^{2n}+1/2\right) x>\frac{1}{2}
\end{equation*}%
and%
\begin{equation*}
\cos \left( 2^{2n}+1/2\right) x\leq \frac{1}{4}.
\end{equation*}%
Consequently, from (\ref{low}) we obtain%
\begin{eqnarray*}
&&R_{1}+R_{2}+R_{3}+R_{4} \\
&\geq &R_{1}+R_{2}-R_{3}+R_{4} \\
&\geq &\frac{1}{4}\frac{1}{xy}\sum\limits_{k=1}^{2^{2n}-2}\frac{1}{k\left(
k+1\right) \left( k+2\right) }\frac{\sin ^{2}\left( \left( k+1\right) \frac{%
x+y}{2}\right) }{2\sin ^{2}\frac{x+y}{2}} \\
&&+\frac{1}{4}\frac{1}{xy}\sum\limits_{k=1}^{2^{2n}-2}\frac{1}{k\left(
k+1\right) \left( k+2\right) }\frac{\sin ^{2}\left( \left( k+1\right) \frac{%
x-y}{2}\right) }{2\sin ^{2}\frac{x-y}{2}} \\
&&-\frac{1}{16}\left( \frac{\pi }{2}\right) ^{2}\frac{1}{xy}%
\sum\limits_{k=1}^{2^{2n}-2}\frac{1}{k\left( k+1\right) \left( k+2\right) }%
\frac{\sin ^{2}\left( \left( k+1\right) \frac{x+y}{2}\right) }{2\sin ^{2}%
\frac{x+y}{2}} \\
&&-\frac{1}{16}\left( \frac{\pi }{2}\right) ^{2}\frac{1}{xy}%
\sum\limits_{k=1}^{2^{2n}-2}\frac{1}{k\left( k+1\right) \left( k+2\right) }%
\frac{\sin ^{2}\left( \left( k+1\right) \frac{x-y}{2}\right) }{2\sin ^{2}%
\frac{x-y}{2}} \\
&\geq &\frac{c}{xy}\sum\limits_{k=1}^{2^{n-1}-1}\frac{1}{k\left( k+1\right)
\left( k+2\right) }\frac{\sin ^{2}\left( \left( k+1\right) \frac{x+y}{2}%
\right) }{2\sin ^{2}\frac{x+y}{2}} \\
&&+\frac{c}{xy}\sum\limits_{k=1}^{2^{n-1}-1}\frac{1}{k\left( k+1\right)
\left( k+2\right) }\frac{\sin ^{2}\left( \left( k+1\right) \frac{x-y}{2}%
\right) }{2\sin ^{2}\frac{x-y}{2}} \\
&\geq &\frac{c}{xy}\sum\limits_{k=1}^{2^{n-1}-1}\frac{1}{k}\geq \frac{cn}{xy}
\end{eqnarray*}%
which complete the proof of Lemma \ref{main}.
\end{proof}

\begin{corollary}
\label{pointwise}Let $n\geq n_{0}$ and $0\leq s,t\leq \gamma _{n}:=\frac{\pi
/2-\arccos \left( 1/4\right) }{4\left( 2^{2n}+1/2\right) }$. Then%
\begin{equation*}
F_{2^{2n}}\left( x-s,y-t\right) \geq \frac{c}{xy},
\end{equation*}%
for%
\begin{eqnarray*}
\left( x,y\right)  &\in &J_{n} \\
&:&=\bigcup\limits_{l,m=1}^{2^{n-3}}\left\{ \left( x,y\right) :\alpha
_{mn}+\gamma _{n}\leq x\leq \beta _{mn}-\gamma _{n},\alpha _{ln}+\gamma
_{n}\leq x\leq \beta _{ln}-\gamma _{n}\right\} .
\end{eqnarray*}
\end{corollary}

\section{Proof of the Theorems}

\begin{proof}[Proof of Theorem \protect\ref{normdivergence}]
a) Let $Q\left( 2^{4n}\right) \geq c2^{4n}$ for $n>n_{0}$. By virtue of
estimate (\cite{K-R}, Ch2)%
\begin{equation*}
\left\Vert f\right\Vert _{L_{Q}\left( T^{2}\right) }\leq c\left(
1+\left\Vert Q\left( \left\vert f\right\vert \right) \right\Vert
_{L_{1}\left( T^{2}\right) }\right)
\end{equation*}%
we get%
\begin{eqnarray}
&&\left\Vert t_{2^{2n}}\left( \left( \frac{\pi /2-\arccos \left( 1/4\right)
}{8}\right) ^{2}\frac{1_{\left[ 0,\gamma _{n}\right] ^{2}}}{\gamma _{n}^{2}}%
\right) \right\Vert _{L_{1}\left( T^{2}\right) }  \label{a} \\
&\leq &\left\Vert t_{2^{2n}}\right\Vert _{L_{Q}\left( T^{2}\right)
\rightarrow L_{1}\left( T^{2}\right) }\left\Vert \left( \frac{\pi /2-\arccos
\left( 1/4\right) }{8}\right) ^{2}\frac{1_{\left[ 0,\gamma _{n}\right] ^{2}}%
}{\gamma _{n}^{2}}\right\Vert _{L_{Q}\left( T^{2}\right) }  \notag \\
&\leq &c\left\Vert t_{2^{2n}}\right\Vert _{L_{Q}\left( T^{2}\right)
\rightarrow L_{1}\left( T^{2}\right) }\left( 1+\left\Vert Q\left( \left(
\frac{\pi /2-\arccos \left( 1/4\right) }{8}\right) ^{2}\frac{1_{\left[
0,\gamma _{n}\right] ^{2}}}{\gamma _{n}^{2}}\right) \right\Vert
_{L_{1}\left( T^{2}\right) }\right)   \notag \\
&\leq &c\left\Vert t_{2^{2n}}\right\Vert _{L_{Q}\left( T^{2}\right)
\rightarrow L_{1}\left( T^{2}\right) }\left( 1+\gamma _{n}^{2}Q\left( \left(
\frac{\pi /2-\arccos \left( 1/4\right) }{8}\right) ^{2}\frac{1}{\gamma
_{n}^{2}}\right) \right)   \notag \\
&\leq &c\left\Vert t_{2^{2n}}\right\Vert _{L_{Q}\left( T^{2}\right)
\rightarrow L_{1}\left( T^{2}\right) }\frac{Q\left( 2^{4n}\right) }{2^{4n}}.
\notag
\end{eqnarray}%
On the other hand, from Corollary \ref{pointwise} we obtain%
\begin{eqnarray*}
&&t_{2^{2n}}\left( \left( \frac{\pi /2-\arccos \left( 1/4\right) }{8}\right)
^{2}\frac{1_{\left[ 0,\gamma _{n}\right] ^{2}}}{\gamma _{n}^{2}};x,y\right)
\\
&=&\left( \frac{\pi /2-\arccos \left( 1/4\right) }{8}\right) ^{2}\frac{1}{%
\gamma _{n}^{2}}\int\limits_{0}^{\gamma _{n}}\int\limits_{0}^{\gamma
_{n}}F_{2^{2n}}\left( x-s,y-t\right) dsdt \\
&\geq &\frac{c}{xy},\text{ \ \ }\left( x,y\right) \in J_{n}.
\end{eqnarray*}%
Consequently,%
\begin{eqnarray}
&&\left\Vert t_{2^{2n}}\left( \left( \frac{\pi /2-\arccos \left( 1/4\right)
}{8}\right) ^{2}\frac{1_{\left[ 0,\gamma _{n}\right] ^{2}}}{\gamma _{n}^{2}}%
\right) \right\Vert _{L_{1}\left( T^{2}\right) }  \label{b} \\
&\geq &c\iint\limits_{J_{n}}\frac{1}{xy}dxdy  \notag \\
&\geq &c\sum\limits_{l,m=1}^{2^{n-3}}\log \left( 1+\frac{c}{m}\right) \log
\left( 1+\frac{c}{l}\right)   \notag \\
&\geq &cn^{2}.  \notag
\end{eqnarray}

The fact that $L_{Q}\left( T^{2}\right) \nsubseteq L\log ^{2}L\left(
T^{2}\right) $ is equivalet to the condition%
\begin{equation*}
\overline{\lim\limits_{u\rightarrow \infty }}\frac{u\log ^{2}u}{Q\left(
u\right) }=+\infty .
\end{equation*}%
Thus there exists $\{u_{k}:k\geq 1\}$ such that
\begin{equation*}
\lim\limits_{k\rightarrow \infty }\frac{u_{k}\log ^{2}u_{k}}{Q\left(
u_{k}\right) }=+\infty ,~u_{k+1}>u_{k},k=1,2,...\
\end{equation*}%
and a monotonically increasing sequence of positive integers $\left\{
r_{k}:k\geq 1\right\} $ such that%
\begin{equation*}
2^{4r_{k}}\leq u_{k}<2^{4\left( r_{k}+1\right) }.
\end{equation*}%
Then we have%
\begin{equation*}
\frac{2^{4r_{k}}r_{k}^{2}}{Q\left( 2^{4r_{k}}\right) }\geq c\frac{u_{k}\log
^{2}u_{k}}{Q\left( u_{k}\right) }\rightarrow \infty \text{ \ \ as \ \ }%
k\rightarrow \infty ,
\end{equation*}%
thus from (\ref{a}) and (\ref{b}) we obtain that%
\begin{equation*}
\sup\limits_{n}\left\Vert t_{2^{2n}}\right\Vert _{L_{Q}\left( T^{2}\right)
\rightarrow L_{1}\left( T^{2}\right) }=+\infty .
\end{equation*}

Part b) follows immediately from part a).

This completes the proof of Theorem \ref{normdivergence}.
\end{proof}

\begin{proof}[Proof of Theorem \protect\ref{measuredivergence}]
By Lemma \ref{GGT} the proof of Theorem \ref{measuredivergence} will be
complete if we show that there exists sequences of integers $\{n_{k}:k\geq
1\}$ and $\{\nu _{k}:k\geq 1\}$ increasing to infinity, and a sequence of
functions $\{\xi _{k}:k\geq 1\}$ from the unit bull $S_{Q}\left( 0,1\right) $
of Orlicz space $L_{Q}\left( T^{2}\right) $, such that for all $k$%
\begin{equation}
\mes\{\left( x,y\right) \in T^{2}:\left\vert t_{2^{2n_{k}}}\left( \xi
_{k};x,y\right) \right\vert >\nu _{k}\}\geq \frac{1}{8}.  \label{est}
\end{equation}

First, we prove that there exists $c_{1},c_{2}>0$ such that
\begin{equation}
\mes\{\left( x,y\right) \in T^{2}:\left\vert t_{2^{2n}}\left( \frac{1_{\left[
0,\gamma _{n}\right] ^{2}}}{\gamma _{n}^{2}};x,y\right) \right\vert
>c_{1}2^{3n}\}>\frac{c_{2}n}{2^{3n}}.  \label{est1}
\end{equation}

From Corollary \ref{pointwise} we can write

\begin{eqnarray}
&&\mes\left\{ (x,y)\in T^{2}:\left\vert t_{2^{2n}}\left( \frac{1_{\left[
0,\gamma _{n}\right] ^{2}}}{\gamma _{n}^{2}};x,y\right) \right\vert
>c_{1}2^{3n}\right\}  \label{est2} \\
&\geq &\mes\left\{ (x,y)\in J_{n}:\left\vert t_{2^{2n}}\left( \frac{1_{\left[
0,\gamma _{n}\right] ^{2}}}{\gamma _{n}^{2}};x,y\right) \right\vert
>c_{1}2^{3n}\right\}  \notag \\
&>&\mes\{(x,y)\in J_{n}:\frac{1}{xy}>2^{3n}\}.  \notag
\end{eqnarray}

Set%
\begin{equation*}
r_{n,m}:=\max \left\{ l:\beta _{ln}\leq \frac{1}{2^{3n}\left( \beta
_{mn}-\gamma _{n}\right) }+\gamma _{n}\right\} .
\end{equation*}%
For the simple calculation we obtain that%
\begin{equation*}
r_{n,m}\sim \frac{2^{n}}{m}.
\end{equation*}

Then from (\ref{est2}) we have%
\begin{eqnarray*}
&&\mes\left\{ (x,y)\in T^{2}:\left\vert t_{2^{2n}}\left( \frac{1_{\left[
0,\gamma _{n}\right] ^{2}}}{\gamma _{n}^{2}};x,y\right) \right\vert
>c_{1}2^{3n}\right\} \\
&\geq &c\sum\limits_{m=1}^{2^{n-12}}\sum\limits_{l=1}^{r_{n,m}}\mes\left(
J_{n}\right) \\
&\geq &\frac{c}{2^{4n}}\sum\limits_{m=1}^{2^{n-12}}r_{n,m}\geq \frac{c}{%
2^{3n}}\sum\limits_{m=1}^{2^{n-12}}\frac{1}{m}\geq \frac{c_{2}n}{2^{3n}}
\end{eqnarray*}

Hence (\ref{est1}) is proved.

From the condition of the theorem we write \cite{K-R}
\begin{equation*}
\liminf_{u\rightarrow \infty }\frac{Q(u)}{u\log u}=0.
\end{equation*}%
Consequently, there exists a sequence of integers $\{n_{k}:k\geq 1\}$
increasing to infinity, such that

\begin{equation}
\lim_{k\rightarrow \infty }Q(2^{4n_{k}})2^{-4n_{k}}n_{k}^{-1}=0,\quad \frac{%
Q(2^{4n_{k}})}{2^{4n_{k}}}\geq 4,\quad \forall k.  \label{cond1}
\end{equation}

From (\ref{est1}) we have
\begin{equation*}
\mes\{(x,y)\in T^{2}:\left\vert t_{2^{2n_{k}}}\left( \frac{1_{\left[
0,\gamma _{n_{k}}\right] ^{2}}}{\gamma _{n_{k}}^{2}};x,y\right) \right\vert
>c_{1}2^{3n_{k}}\}>\frac{c_{2}n_{k}}{2^{3n_{k}}}.
\end{equation*}%
Then by the virtue of Theorem \ref{Garsia} there exists $%
E_{1},...,E_{r},E_{1}^{\prime },...,E_{r}^{\prime }\in \mathcal{E}$ and $%
\varepsilon _{1},...,\varepsilon _{r}=\pm 1$ such that
\begin{equation*}
\mes\{(x,y)\in T^{2}:\left\vert \sum\limits_{i=1}^{r}\varepsilon
_{i}t_{2^{2n_{k}}}\left( \frac{1_{\left[ 0,\gamma _{n_{k}}\right] ^{2}}}{%
\gamma _{n_{k}}^{2}};E_{i}x,E_{i}^{\prime }y\right) \right\vert
>2^{3n_{k}}\}>\frac{1}{8},
\end{equation*}%
where $r=\left[ \frac{2^{3n_{k}}}{c_{2}m_{k}}\right] +1.$

Denote
\begin{equation*}
\nu _{k}=\frac{2^{7n_{k}-1}}{rQ\left( 2^{4n_{k}}\right) }
\end{equation*}%
and
\begin{equation*}
\xi _{k}\left( x,y\right) =\frac{2^{4n_{k}-1}}{Q\left( 2^{4n_{k}}\right) }%
M_{k}\left( x,y\right) ,
\end{equation*}%
where
\begin{equation*}
M_{k}\left( x,y\right) =\frac{1}{r}\sum\limits_{i=1}^{r}\varepsilon _{i}%
\frac{1_{\left[ 0,\gamma _{n_{k}}\right] ^{2}}\left( E_{i}x,E_{i}^{\prime
}y\right) }{\gamma _{n_{k}}^{2}}.
\end{equation*}

Thus, we obtain (\ref{est}).

Since%
\begin{equation*}
\left\Vert M_{k}\right\Vert _{\infty }\leq 2^{4n_{k}+2},
\end{equation*}%
\begin{equation*}
\left\Vert M_{k}\right\Vert _{1}\leq 1,
\end{equation*}%
\begin{equation*}
\left\vert \iint\limits_{T^{2}}fg\right\vert \leq \frac{1}{2}\left[
\iint\limits_{T^{2}}Q\left( 2\left\vert f\right\vert \right) +1\right] ,
\end{equation*}%
and%
\begin{equation*}
\frac{Q\left( u\right) }{u}<\frac{Q\left( u^{\prime }\right) }{u^{\prime }}%
,\left( 0<u<u^{\prime }\right)
\end{equation*}%
we can write

\begin{eqnarray*}
\Vert \xi _{k}\Vert _{L_{Q}(T^{2})} &\leq &\frac{1}{2}\left[
1+\iint\limits_{T^{2}}Q\left( \frac{2^{4n_{k}}\left\vert M_{k}\left(
x,y\right) \right\vert }{Q(2^{4n_{k}})}\right) dxdy\right]  \\
&\leq &\frac{1}{2}\left[ 1+\iint\limits_{T^{2}}\frac{Q\left( \frac{%
2^{4n_{k}}2^{4n_{k}+2}}{Q(2^{4m_{k}})}\right) }{\frac{2^{4n_{k}}2^{4n_{k}+2}%
}{Q(2^{4m_{k}})}}\frac{2^{4n_{k}}\left\vert M_{k}\left( x,y\right)
\right\vert }{Q(2^{4n_{k}})}dxdy\right]  \\
&\leq &\frac{1}{2}\left[ 1+\iint\limits_{T^{2}}\frac{Q\left(
2^{4n_{k}}\right) }{2^{4n_{k}}}\frac{2^{4n_{k}}\left\vert M_{k}\left(
x,y\right) \right\vert }{Q(2^{4n_{k}})}dxdy\right]  \\
&\leq &1.
\end{eqnarray*}

Hence, $\xi _{k}\in S_{Q}\left( 0,1\right) $, and Theorem \ref%
{measuredivergence} is proved.
\end{proof}

The validity of Corollary \ref{measuredivergence2} follows immediately from
Theorem \ref{measuredivergence} and Lemma \ref{GGT2}.

\end{document}